\documentclass[12pt,a4paper,twoside]{article}

\usepackage[centertags]{amsmath}
\usepackage{amsfonts}
\usepackage{amssymb}
\usepackage{amsthm}
\usepackage{ulem}
\usepackage{epsfig}
\usepackage{subfigure}
\usepackage{color}
\usepackage{mathrsfs}
\usepackage[boldsans]{ccfonts}
\usepackage{array}
\usepackage{calc}
\usepackage{titlesec}
\usepackage{graphicx}
\usepackage{enumitem}

\newtheoremstyle{break}
  {\topsep}{\topsep}%
  {\itshape}{}%
  {\bfseries}{}%
  {\newline}{}%
\newtheoremstyle{break1}
  {\topsep}{\topsep}%
  {}{}%
  {\bfseries}{}%
  {\newline}{}%
\theoremstyle{break} \newtheorem{theorem}{Theorem}[section]
\theoremstyle{break} 
\theoremstyle{break1} \newtheorem{remark}[theorem]{Remark}
\theoremstyle{break} 
\theoremstyle{break1} \newtheorem{definition}[theorem]{Definition} 
\theoremstyle{break} \newtheorem{lemma}[theorem]{Lemma}
\theoremstyle{break} \newtheorem{corollary}[theorem]{Corollary}
\theoremstyle{break} 
\theoremstyle{break} 
\theoremstyle{break} 
\theoremstyle{break} 
\theoremstyle{break} \newtheorem{proposition}[theorem]{Proposition}
\theoremstyle{break} 
\numberwithin{equation}{section}

\newcommand{\hide}[1]{}

\newcommand{\R}{{\mathbb{R}}}

\newcommand{\C}{{\mathbb{C}}}

\newcommand{\B}{{\mathbb{B}}}

\def\Re{\mathop{{\rm Re}}}

\def\HB{\mathop{{\rm Hol}}(\B^n,\C^n)}
\def\M{\mathcal{M}_n}
\def\id{\mathop{{\rm id}}}
\def\supp{\mathop{{\rm supp}}}
\def\ext{\mathop{{\rm ext}}}

\newcommand{\eps}{\varepsilon}

\begin{document}

\vspace*{-1cm}

\renewcommand{\thefootnote}{}
\begin{center}
{\Large \bf Pontryagin's maximum principle for\\[2mm]
 the Loewner equation  in higher dimensions}\\[5mm]
{\large Oliver Roth}\\
\today\\[2mm]
\end{center}

\footnotetext{Mathematics Subject Classification (2000) \quad Primary 32H02
  $\cdot$ 30C55 $\cdot$ 49K15}
\begin{abstract}
\noindent In this paper we develop a variational method for the Loewner
equation in higher dimensions. As a result we obtain a version of Pontryagin's
maximum principle from optimal control theory for the Loewner equation in
several complex variables. Based on recent work of Arosio, Bracci and Wold
\cite{ABFW}, %work of Docquier and Grauert on the Levi problem in complex geometry, 
we then apply  our version of the Pontryagin maximum
principle to obtain first--order necessary conditions for the extremal
functions for a wide class
of extremal problems over the set of  normalized biholomorphic mappings on the unit ball in $\C^n$.
\end{abstract}

\section{Introduction}
Let $\B^n:=\{z \in \C^n \, : \, ||z||<1\}$ denote the unit ball of $\C^n$ 
with respect to the euclidean norm $||\cdot||$ and let
$\HB$ be
 the vector space of all holomorphic maps from $\B^n$ into $\C^n$. The set
 $$\mathcal{S}_n:=\{f \in \HB \, : \, f(0)=0, df_0=\id, f \text{ univalent}\}$$
of normalized biholomorphic mappings on $\B^n$ has been 
introduced by H.~Cartan \cite{Car33}.
 One of the main problems when dealing with univalent functions in the class
 $\mathcal{S}_n$ in dimensions $n>1$ is the fact that there is no Riemann mapping
 theorem available. 
In particular, this makes it fairly difficult to construct variations of a
given map in the class $\mathcal{S}_n$. 

\medskip

The aim of the present paper is to develop a variational
method which works effectively for univalent functions that can be obtained as
solutions of Loewner--type differential equations.
We present the details only for the class $\mathcal{S}_n^0 \subset \mathcal{S}_n$ of
all functions which admit a so--called parametric representation by means of
the Loewner  equation.
This class has been introduced by I.~Graham, G.~Kohr et al.~(see
e.g.~\cite{GK,GHKK07}) and is obtained in a most natural way by generalizing the
  classical one--dimenional Loewner equation \cite{Loe23} to higher dimensions.
We note that the approach of the present paper can also be used for other
more general Loewner--type equations, e.g.~for the class of functions that
have a so--called $A$--parametric representation (see
\cite{GHKK07,DGHK,GHKK13}), and  also 
 for the various Loewner equations in the unit disk and complete
hyperbolic manifolds, which have recently been studied intensively (see \cite{ABCD, A, AB, ABHK, ABFW, BCD09, BCD12}).

\medskip

We now give a short account of the results of this paper and start by
introducing some notation.

\begin{definition}
Let $$\M:=\left\{ h \in \HB \, : \, h(0)=0, dh_0=-\id, \, \Re \langle
  h(z),z\rangle \le 0 \text{ for all } z \in \B^n \right\} \, .$$
Here, $\langle \,\cdot,\cdot \rangle$ denotes the standard Euclidean inner product of $\C^n$.
 A \textit{Herglotz vector field in the class $\M$} is a mapping $G : \B^n
 \times \R^+ \to \C^n$ such that
\begin{itemize}
\item[(i)] $G(z,\cdot)$ is measurable on $\R^+$ for every $z \in \B^n$, and
\item[(ii)] $G(\cdot,t) \in \M$ for a.e.~$t \in \R^+$.
\end{itemize}
\end{definition}

It is not difficult to show that
a function $h \in \HB$ satisfying $h(0)=0$ and $dh_0=-\id$ belongs to $\M$ 
if and only if $\Re \langle -h(z),z \rangle>0$ for all $z \in \B^n \backslash \{0\}$,
see \cite[Remark 2.1]{BHKG}. In particular, 
the set $\M$ is exactly the class $-\mathcal{M}$ as defined e.g.~in
\cite[p.~203]{GK}. Hence, $\M$ is a compact subset of $\HB$ (see \cite[Theorem
6.1.39]{GK}). This fact will play an important r$\hat{\text{o}}$le in this paper.

\begin{definition}[The Loewner Equation on the unit ball $\B^n$]
Let $G(z,t)$ be a Herglotz vector field in the class $\M$.
We denote by $\varphi^G_{t}$  the unique
solution $\varphi_{t}$ of the Loewner ODE
\begin{equation} \label{eq:L} \begin{array}{rcl}
\dot{\varphi}_{t} (z)&=& G(\varphi_{t}(z),t) \quad \text{ for a.e. } t \ge 0 \,
, \\[2mm]
\varphi_0(z)&=&z \in \B^n\, .
\end{array}
\end{equation}
\end{definition}

For any Herglotz vector field $G(z,t)$ in the class $\M$, the limit
$$f^G:=\lim_{t \to \infty} e^t \varphi^G_t$$ exists locally uniformly in
$\B^n$ and belongs to $\mathcal{S}_n$,
see \cite[Thm.~8.1.5]{GK}.  We can therefore define
$$ \mathcal{S}_n^0:=\left\{f \, | \, f=f^G \text{ for some Herglotz vector
  field } G \text{ in the class } \M \right\} \, .$$
It is known that the class $\mathcal{S}_n^0$ is compact and that 
 $e^t\varphi_t^G \in \mathcal{S}_n^0$ for all
$t \in \R^+_0$ and every Herglotz vector field $G$ in the class $\M$, see \cite{GK}.
Hence one may think of $\mathcal{S}_n^0$ as the ``reachable set'' of the Loewner
equation (\ref{eq:L}).

\begin{theorem}[A variational formula in $\mathcal{S}_n^0$] \label{thm:main1}
Let $f \in \mathcal{S}_n^0$. Suppose that $G(z,t)$ is a Herglotz vector field in the
class $\M$  such that $f=f^G$. Then  for almost every $t\ge 0$
and any $h \in \M$
there exists a family of functions $f^{\eps} \in \mathcal{S}_n^0$ such that
$$ f^{\eps}(z)=f(z)+\eps \, d(f)_z \cdot \left[ d(\varphi^G_t)_z \right]^{-1}
\left[ h\left(\varphi^G_t(z)\right)-G\left(\varphi^G_t(z),t\right)
\right]+r^{\eps}(z) \, .$$
Here, the error term $r^{\eps} \in \HB$ has the property that $r^{\eps}/\eps
\to 0$ locally uniformly in $\B^n$ as $\eps \to0+$.
\end{theorem}

The variations $f^{\eps}$ in Theorem \ref{thm:main1} will be constructed with 
help of ``spike variations''. This is a well--known method in control theory and the
calculus of variations which goes back at least to Weierstra{\ss}. In proving
 Theorem \ref{thm:main1} we shall show  that it is possible to
 modify this technique in such a way that it can be applied for the infinite--dimensional 
Fr\'echet space $\HB$ (endowed with the standard compact--open topology).
We note that a different variational technique in $\mathcal{S}_n^0$
 has recently been  developed by Bracci, Graham, Hamada and Kohr \cite{BHKG}. Theorem \ref{thm:main1} has the advantage that it works for
\textit{any} function $f \in \mathcal{S}_n^0$, while the method of \cite{BHKG}
is restricted to those functions in $\mathcal{S}_n^0$ which can be embedded in a
so--called ``ger\"aumig'' Loewner chain, see \cite{BHKG} for details.

\medskip

One main field of application of the variational formula of Theorem
\ref{thm:main1} is the study of extremal problems in the class $\mathcal{S}_n^0$.
We call a function $F \in \mathcal{S}_n^0$ an extremal function for a
functional $\Phi : \mathcal{S}_n^0 \to \C$
if $\Re \Phi(f)\le\Re \Phi(F)$ for every $f \in \mathcal{S}_n^0$. 
Here and henceforth we assume that the
functional $\Phi : \mathcal{S}_n^0 \to \C$ is  complex differentiable
 in the sense of R.~Hamilton's Fr\'echet space calculus
as developed in \cite{Ham82} 
(see Definition \ref{def:der} below for details).

\begin{theorem} \label{thm:main2a}
Let $F \in \mathcal{S}_n^0$ be an extremal function for a functional $\Phi :
\mathcal{S}_n^0 \to \C$  with complex derivative $L$ at $F$.
 Suppose that $G(z,t)$ is a Herglotz vector field in the
class $\M$ such that $F=f^G$.  For each $t \ge 0$ let $L_t$ be the
continuous linear functional  on $\HB$ defined by
$$ L_t(h):=L \left( d(F)_z \cdot \left[ d (\varphi^G_t)_z
  \right]^{-1} \cdot h\left(\varphi^G_t\right) \right) \, , \qquad h \in \HB \, .$$
Then for a.e.~$t \ge 0$, 
$$ \Re L_t(h) \le \Re L_t(G(\cdot,t)) \quad \text{ for all } h \in \M \, .$$
\end{theorem}

Theorem \ref{thm:main2a} is in fact a version of Pontryagin's maximum principle
for the case of the Loewner equation in higher dimensions. 
It generalizes earlier well--known work on control theory of the Loewner equation in one
dimension which has been initiated by Goodman \cite{Goodman},
Popov \cite{Popov} and Friedland \&
Schiffer \cite{FS0,FS1},  and which has been developed into a powerful theory by D.~Prokhorov \cite{Pro92,Pro93,Pro02}, see also  \cite{Ro}.

\medskip

At first sight it is not clear that Pontryagin's maximum principle (Theorem \ref{thm:main2a}) carries
any useful information about the Herglotz vector field $G(\cdot,t)$ at all, 
simply because the linear functionals $L_t$ in Theorem
\ref{thm:main2a} might be constant on the class $\M$. In particular,
Theorem \ref{thm:main2a} alone is not sufficient to deduce that $G(\cdot,t)$ is a
support point (see Definition \ref{def:supp}) in the class $\M$.
However, refering to a deep result of Docquier and Grauert \cite{DG}, it has
recently been observed by Arosio, Bracci and Wold \cite{ABFW} that all domains
$\varphi_t^G(\B^n)$ are \textit{Runge domains}. Using this Runge property 
we shall show in Proposition \ref{prop:lin} below that if $L$ is not constant
on $\mathcal{S}^0_n$, then 
$L_t$ is
\textit{never} constant on $\M$. In combination with 
Pontryagin's maximum principle in the form of Theorem \ref{thm:main2a}, we are
therefore led to the following necessary  condition for extremal problems in the class $\mathcal{S}_n^0$.

\begin{theorem} \label{thm:main3a}
Let $F \in \mathcal{S}_n^0$ be an extremal function for a functional
$\Phi : \mathcal{S}_n^0 \to \C$ with complex derivative $L$ at $F$. Suppose that
$L$ is not constant on $\mathcal{S}_n^0$.
If $G(z,t)$ is a Herglotz vector field
in the class $\M$  such  that $F=f^G$, then 
$G(\cdot,t)$ is a support point in the class $\M$ for a.e.~$t \ge 0$.
\end{theorem}

Roughly speaking, Theorem \ref{thm:main3a} says that if a Herglotz vector
field $G(z,t)$ generates an extremal function in the class 
$\mathcal{S}^0_n$ via the Loewner equation, then for a.e.~$t \ge 0$ the function
$G(\cdot,t) \in \M$ itself has to be extremal in the class $\M$.
As an illustration of the use of Theorem \ref{thm:main3a} we prove 
in Corollary \ref{cor:main} below a generalization of a recent result
due to Bracci, Graham, Hamada and Kohr \cite{BHKG} about support points in $\mathcal{S}_n^0$.

\medskip

We finally point out another consequence of Theorem \ref{thm:main1}.

\begin{theorem} \label{thm:pommerenke}
Let $F \in \mathcal{S}_n^0$ be an extremal function for a functional $\Phi :
\mathcal{S}_n^0 \to \C$ with complex derivative $L$ at $F$. Then
$$ \max \limits_{h \in \mathcal{M}_n} \Re L(d(F)_z \cdot  h)=-\Re L(F) \,
.$$ 
\end{theorem}

Theorem \ref{thm:pommerenke} extends a result of Pommerenke (see
\cite[p.~185]{Pom}), which deals with the case of dimension $n=1$ (and
functionals of finite degree), to the cases $n>1$ and arbitrary complex
differentiable functionals. We note that the case $n=1$ allows a fairly
elementary proof, which is based on the 
``lucky accident'' (see \cite[p.~231]{Dur83})
 that  the Koebe functions
$$k_{\zeta}(z):=\frac{z}{\left(1+\overline{\zeta} z\right)^2} \, , \qquad \zeta \in \partial
\B^1 \, , $$
generate the set $\ext \mathcal{M}_1$ of  extreme points of $\mathcal{M}_1$ via
$$ \ext \mathcal{M}_1=\left\{ -z \frac{\zeta+z}{\zeta-z} \, : \, \zeta
  \in \partial \B^1 \right\}=
\bigg\{ -\left[ d\left( k_{\zeta} \right)_z \right]^{-1}
  \cdot k_{\zeta}(z)  \, : \, \zeta \in \partial \B^1 \bigg\} \, .$$
For $n>1$, however, the set $\ext \M$ of  extreme points of $\M$ is not known
(see \cite{Voda} for recent results in this direction), so we employ a
completely different approach for the proof of Theorem \ref{thm:pommerenke}.

\medskip

This paper is organized in the following way. We start in Section \ref{sec:2}
by constructing variations of evolution families for the Loewner
equation in higher dimensions. In Section \ref{sec:3} we generalize this result
to produce variations in the class $\mathcal{S}_n^0$ and we prove Theorem
\ref{thm:main1}. We also produce variations of parametric representations,
which partly extend the recent results in \cite{BHKG}. 
In the final Section \ref{sec:4} we apply the results of Sections \ref{sec:2}
and \ref{sec:3} to study extremal problems in the class $\mathcal{S}_n^0$ and
we prove  Theorem \ref{thm:main2a}, Theorem \ref{thm:main3a} and Theorem \ref{thm:pommerenke}.

\section{Variations of evolution families} \label{sec:2}

In this section, we construct variations of  Loewner evolution families.

\begin{definition}
Let $G(z,t)$ be a Herglotz vector field in the class $\M$. Denote
for fixed $s \ge 0$ by $\varphi^G_{s,t}$ the solution to
  \begin{equation} \label{eq:L2} \begin{array}{rcl}
\dot{\varphi}_{s,t} (z)&=& G(\varphi_{s,t}(z),t) \quad \text{ for a.e. } t \ge s \,
, \\[2mm]
\varphi_{s,s}(z)&=&z \in \B^n\, .
\end{array}
\end{equation}
We call $(\varphi^G_{s,t})_{0 \le s \le t}$ \textit{the evolution family
  generated by $G(z,t)$}. 
\end{definition}

\begin{lemma} \label{lem:lebesgue}
Let $G(z,t)$ be a Herglotz vector field in the class $\M$.
Then there exists a set $E_G \subseteq \R^+$ of zero measure such that
for all $t \in (0,\infty) \backslash E_G$ the condition
\begin{equation} \label{eq:lebesgue}
 G(z,t)=\lim \limits_{\eps \to 0+} \frac{1}{\eps} \int \limits_{t-\eps}^t G(z,\tau) \, d\tau
\, 
\end{equation}
holds locally uniformly w.r.t.~$z \in \B^n$. 
\end{lemma}

\begin{proof} We fix $z \in \B^n$. Since $G(\cdot, t) \in \M$ for a.e.~$t>0$
  and $\M$ is a compact subset of $\HB$, the measurable function
$t \mapsto G(z,t)$ is (essentially) bounded on the interval
$(0,\infty)$. Therefore,  there exists a set $E_G(z) \subseteq \R^+$ of zero measure such that
condition (\ref{eq:lebesgue}) holds for all $t \in \R^+ \backslash E_G(z)$.
Now choose a dense countable set $A \subseteq \B^n$ and set $E_G:=\cup_{a \in
  A} E_G(a)$. Then $E_G$ has zero measure and (\ref{eq:lebesgue}) holds
for every $t \in \R^+ \backslash E_G$ and every point $z$ 
 in the dense subset $A \subseteq \B^n$. Since $\M$ is
a normal family and $G(\cdot,t) \in \M$ for a.e.~$t \ge 0$, this implies that
(\ref{eq:lebesgue}) holds locally uniformly in $\B^n$ for every fixed $t \in
\R^+\backslash E_G$ by Vitali's theorem.
\end{proof}

\begin{remark} \label{rem:lebesgue}
We call the set $R_G:=\R^+\backslash E_G$ the \textit{regular set of the
  Herglotz vector field $G(z,t)$} and every $T \in
R_G$ is called a \textit{regular point for $G(z,t)$}.
Note that if $T$ is a regular point for $G(z,t)$ and $\varphi_{s,t}:=\varphi^G_{s,t}$, then
Lemma \ref{lem:lebesgue} implies that for any $s <T$,
$$  G(\varphi_{s,T}(z),T)=\lim \limits_{\eps \to 0+} \frac{1}{\eps} \int
\limits_{T-\eps}^T G(\varphi_{s,\tau}(z),\tau) \, d\tau \, $$
locally uniformly w.r.t.~$z \in \B^n$,
since $\varphi_{s,\cdot}(z)$ is absolutely continuous on compact intervals of
$\R^+_0$ locally uniformly w.r.t.~$z \in \B^n$.
\end{remark}

We can now state the main result of this section.

\begin{theorem} \label{thm:main}
Let $G(z,t)$ be a Herglotz vector field in the class $\M$ with
associated evolution family $\varphi_{s,t}:=\varphi^{G}_{s,t}$ and let $T \in
R_G$ be a regular point. Then for any  $h \in \M$ and any $\eps \in (0,T)$
there exists an  evolution family $(\varphi^{\eps}_{s,t})_{0 \le s \le t}$ such that
$$ \varphi^{\eps}_{s,t}=\varphi_{s,t}+\eps \, \alpha^h_{s,t} +o^{\eps}_{s,t}\, , $$
where
$$ \alpha^h_{s,t}=\begin{cases} 0 & \text{ if } s \le t <T-\eps \text{ or }
  T \le s \le t \, ,\\[2mm]
d(\varphi_t)_z \cdot \left[
  d(\varphi_T)_z \right]^{-1}\cdot  \big[ h(\varphi_T)-G(\varphi_T,T)
\big]  & \text{ if } s < T \le t \, .
\end{cases} \, $$
Here, $o^{\eps}_{s,t}\in \HB$ indicates a term such that
$$ \lim \limits_{\eps \to 0+} \frac{o^{\eps}_{s,t}}{\eps}=0 \quad \text{
  locally uniformly in } \B^n $$
for any fixed $s,t$ such that $s<T \le t$.
\end{theorem}

In order to prove Theorem \ref{thm:main}
we are going to adapt the standard method of 
\textit{needle or spike variations} for the particular
case of the Loewner equation (\ref{eq:L2}).

\begin{definition}[Needle variations]
Let $G(z,t)$ be a Herglotz vector field in the class $\M$, $h \in \M$ and $T>0$.
For each $\eps \in (0,T)$ let 
$$ G_{\eps}(\cdot,t):=G_{\eps,h,T}(\cdot,t):=\begin{cases} 
G(\cdot,t) & \text{ if } \quad t \in \R^+ \backslash (T-\eps,T)\, ,\\
h   & \text{ if } \quad t \in (T-\eps,T) \, .
\end{cases} $$

\begin{figure}[h]
\centerline{\includegraphics[width=7cm]{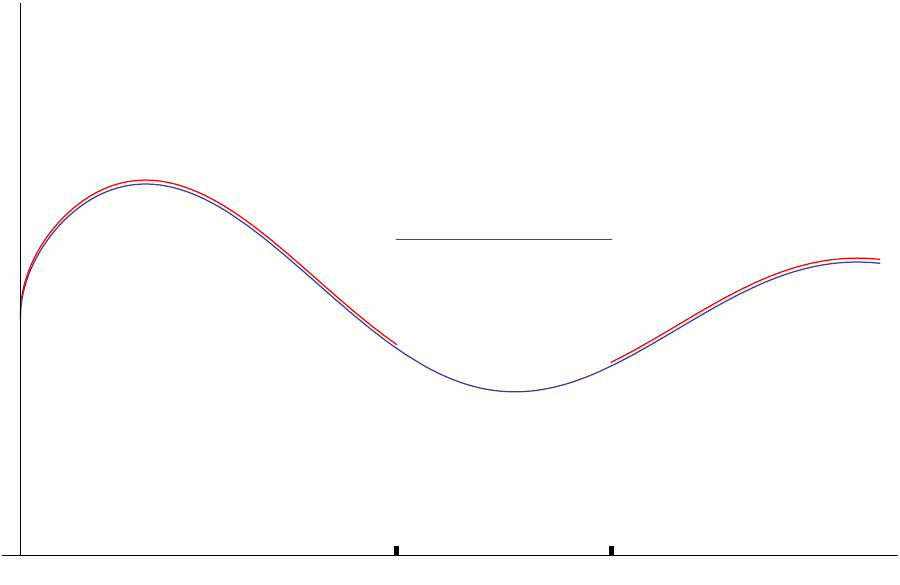}}

\vspace*{-0.4cm} \hspace{6.7cm} $T-\eps$ 
\hspace{0.75cm} $T$

\vspace*{-5.5cm} \hspace{4cm} $\M$ 

\vspace*{4.8cm}
\caption{The graphs of $G(\cdot,t)$ (in blue) and $G_{\eps}(\cdot,t)$ (in red).}
\end{figure}

We call  the Herglotz vector fields $G_{\eps}(z,t)$ in the class $\M$
the \textit{needle variations of $G(z,t)$ with data $(T,h)$}. 
We  also call the evolution families $(\varphi^{\eps}_{s,t}):=(\varphi^{G_{\eps}}_{s,t})$ the \textit{needle
variations of the evolution family $(\varphi^{G}_{s,t})$ with data $(T,h)$}. 
\end{definition}

\begin{remark} \label{rem:s}
Let $G(z,t)$ be a Herglotz vector field in the class $\M$, $h \in \M$ and $T>0$.
Since $G_{\eps}(\cdot,t)=G(\cdot,t)$ for any $t \not \in (T-\eps,T)$, we
immediately get that
$$ \varphi^{\eps}_{s,t}=\varphi_{s,t} \quad \text{ if } \, s \le t \le T-\eps \text{ or
} T \le s \le t \, .$$
In particular, we have
\begin{equation} \label{eq:lolo}
\varphi^{\eps}_{s,t}(z)=\varphi^{\eps}_{s,T-\eps}(z)+\int \limits_{T-\eps}^t
G_{\eps}(\varphi^{\eps}_{s,\tau}(z),\tau) \, d\tau=
\varphi_{s,T-\eps}(z)+\int \limits_{T-\eps}^t
h\left(\varphi^{\eps}_{s,\tau}(z) \right) \, d\tau 
\end{equation}
if $s \le T-\eps \le t \le T$.
\end{remark}

In what follows we use the notation $\overline{\B}_r^n:=\{z \in \C^n \, : \,
  ||z|| \le r\}$. 

\begin{lemma}[Convergence of needle variations] \label{lem:1}
Let $G(z,t)$ be a Herglotz vector field in the class $\M$, $h \in \M$ and $T>0$. Denote by $(\varphi^{\eps}_{s,t})$ 
the needle variations of $(\varphi_{s,t}):=(\varphi^G_{s,t})$ with data $(T,h)$.
Then
for fixed $s \ge 0$, we have
$$\lim \limits_{\eps \to 0+} \varphi_{s,t}^{\eps}(z) = \varphi_{s,t}(z)$$
uniformly for  $(z,t) \in \overline{\B}^n_r \times [s,\infty)$ for any $r \in
(0,1)$. 
\end{lemma}

\begin{proof} In view of Remark \ref{rem:s},
we may assume $s<T$. Fix $r \in (0,1)$. Since $\M$ is compact,
  there is a constant $L_r>0$ such that $||g(z)-g(z')|| \le L_r ||z-z'||$ for any $g \in
  \M$ and every $z,z'\in \overline{\B}^n_r$, see \cite[p.~298]{GK}. For every $t \in[s,T]$ we have $||\varphi_{s,t}(z)||
  \le ||z||$ and therefore we get from the identity (\ref{eq:lolo})  and the fact that
  $\varphi_{s,t}$ is a solution to the evolution equation (\ref{eq:L2}) the
  following estimate
\begin{eqnarray*}
||\varphi^{\eps}_{s,t}(z)-\varphi_{s,t}(z)|| & =  &\\
&&  \hspace*{-3cm}
=\left| \left| \varphi_{s,T-\eps}(z)+\int
\limits_{T-\eps}^t
h(\varphi^{\eps}_{s,\tau}(z)) \,d\tau -\varphi_{s,T-\eps}(z)- \int
\limits_{T-\eps}^t G(\varphi_{s,\tau}(z),\tau) \, d\tau \right|\right| \\
&&  \hspace*{-3cm} =
\left|\left| \, \int
\limits_{T-\eps}^t
h(\varphi^{\eps}_{s,\tau}(z))-G(\varphi_{s,\tau}(z),\tau) \, d\tau \right|\right|\\
& & \hspace*{-3cm} \le \int
\limits_{T-\eps}^t||
h(\varphi^{\eps}_{s,\tau}(z))-G(\varphi^{\eps}_{s,\tau}(z),\tau)|| \, d\tau+
 \int
\limits_{T-\eps}^t||
G(\varphi^{\eps}_{s,\tau}(z),\tau)-G(\varphi_{s,\tau}(z),\tau)|| \, d\tau\\
& & \hspace*{-3cm} \le 2 L_r || z|| (t-T+\eps)+L_r \int \limits_{T-\eps}^t ||
\varphi^{\eps}_{s,\tau}(z)-\varphi_{s,\tau}(z)|| \, d\tau \, .
 \end{eqnarray*}
Using the well--known Gronwall lemma (see \cite[p.~198]{FR}), this implicit
estimate for $||\varphi^{\eps}_{s,t}(z)-\varphi_{s,t}(z)||$ leads to the
explicit estimate
\begin{equation} \label{eq:4} ||\varphi^{\eps}_{s,t}(z)-\varphi_{s,t}(z)|| \le
2 L_r \eps ||z|| \, \left(1+ L_r \eps e^{L_r \eps}\right) \quad \text{ for
  every } t \in [s,T] \, . \end{equation}
In view of the semigroup property $\varphi^{\eps}_{T,t} \circ
\varphi^{\eps}_{s,T}=\varphi^{\eps}_{s,t}$,  we therefore get for all $t > T$,
\begin{equation} \label{eq:5}
 ||\varphi^{\eps}_{s,t}(z)-\varphi_{s,t}(z)||=||
\varphi_{T,t}(\varphi^{\eps}_{s,T}(z))-\varphi_{T,t}(\varphi_{s,T}(z))||\le
C_r ||\varphi^{\eps}_{s,T}(z)-\varphi_{s,T}(z)|| \, ,
\end{equation}
where $C_r>0$ is a constant such that $||\varphi_{s,t}(z)-\varphi_{s,t}(z')||\le C_r
||z-z'||$ for all $t \ge s$ and all $z,z' \in \overline{\B}^n_r$. 
If we combine (\ref{eq:5}) with (\ref{eq:4}), we finally have
\begin{equation} \label{eq:6}
||\varphi^{\eps}_{s,t}(z)-\varphi_{s,t}(z)|| \le \gamma_r \eps \quad \text{ for
  all } ||z|| \le r \text{ and all }  t \ge s \, ,
\end{equation}
where $\gamma_r$ depends only on $r$. This completes the proof of Lemma \ref{lem:1}.
\end{proof}

Lemma \ref{lem:1} says that the needle variations $(\varphi^{\eps}_{s,t})$ of 
$(\varphi_{s,t})$ with data $(T,h)$ form a ``continuous deformation''
of the evolution family $(\varphi_{s,t})$. If $T \in R_G$ is in addition a
\textit{regular} point of $G(z,t)$, then this deformation is actually ``differentiable'' in the
following sense.

\begin{theorem} \label{thm:varpar}
Let $G(z,t)$ be a Herglotz vector field in the class $\M$, let $T \in
R_G$ and $h \in \M$. For fixed $s \in [0,T]$ denote by $\varphi^{\eps}_{s,t}$ 
the needle variations of $\varphi_{s,t}:=\varphi^G_{s,t}$ with data $(T,h)$. Then
$$ \varphi^{\eps}_{s,t}=\varphi_{s,t}+\eps \, d\left(\varphi_{s,t}\right)_z \cdot \left[
  d\left(\varphi_{s,T}\right)_z \right]^{-1} \cdot \big[ h(\varphi_{s,T})-G(\varphi_{s,T},T)
\big] +o^{\eps}_{s,t} \, . $$
for any $t \ge T$. Here, 
 $o^{\eps}_{s,t}$ indicates a term, which divided by $\eps$, tends to $0$
 locally uniformly in $\B^n$ for each fixed $t \ge  T$ as $\eps\to 0+$. 
\end{theorem}

\begin{proof} 
Using (\ref{eq:lolo}), we have
\begin{eqnarray*}
 \frac{\varphi^{\eps}_{s,T}(z)-\varphi_{s,T}(z)}{\eps}
&=&  \frac{\varphi^{\eps}_{s,T}(z)-\varphi_{s,T-\eps}(z)}{\eps}- \frac{\varphi_{s,T}(z)-\varphi_{s,T-\eps}(z)}{\eps}\\
&=&
\frac{1}{\eps} \int  \limits_{T-\eps}^T
h(\varphi^{\eps}_{s,\tau}(z)) \, d\tau-
\frac{1}{\eps} \int \limits_{T-\eps}^T
G(\varphi_{s,\tau}(z),\tau) \, d\tau \, .
%\frac{\varphi_{s,T}(z)-\varphi_{s,T-\eps}(z)}{\eps} \, .
\end{eqnarray*}

%Using Lemma \ref{lem:1} and  Lemma \ref{lem:2}, we see that 
Since $T \in R_G$, we therefore see by using Remark \ref{rem:lebesgue} and
Lemma \ref{lem:1} that
$$ \frac{\partial^+ \varphi^{\eps}_{s,T}(z)}{\partial \eps} \bigg|_{\eps=0} :=\lim
\limits_{\eps \to 0+}
\frac{\varphi^{\eps}_{s,T}(z)-\varphi_{s,T}(z)}{\eps}=h(\varphi_{s,T}(z))-G(\varphi_{s,T}(z),T)
\, ,
$$
where the limit exists locally uniformly in $\B^n$. This proves the claim for
$t=T$. We can now handle the general case $t \ge T$.
By what we have just proved, we know that
$\varphi_{s,t}^{\eps}$ is a solution to 
\begin{equation} \label{eq:var1} \begin{array}{rcl}
 \dot{\varphi}^{\eps}_{s,t}(z) &=& G(\varphi_{s,t}^{\eps}(z),t) \, , \qquad t \ge T \, \\[2mm]
\varphi^{\eps}_{s,T}(z)&=&\varphi_{s,T}(z)+\eps \big[h(\varphi_{s,T}(z))-G(\varphi_{s,T}(z),T)\big]+r_{\eps}(z)
\, ,
\end{array}
\end{equation}
where $r_{\eps} \in \HB$ such that $r_{\eps}/\eps \to 0$ locally
uniformly in $\B^n$ as $\eps\to0+$. We now make use of a standard result
from ODE--theory about ``differentiability with respect to initial conditions''
and differentiate  (\ref{eq:var1}) with
respect to $\eps$, see \cite[Theorem 1A, p.~57]{LM}.
 This way, we find that $$\psi_t(z):=\frac{\partial^+
  \varphi^{\eps}_{s,t}(z)}{\partial \eps} \bigg|_{\eps=0}$$ is a solution to
the 
initial value problem
\begin{equation} \label{eq:var2} \begin{array}{rcl}
 \dot{\psi}_t(z)&=& \displaystyle \frac{\partial G}{\partial z}(\varphi_{s,t}(z),t) \cdot \psi_t(z) \, , \qquad t \ge T \, , \\[3mm]
\psi_T(z)&=& h(\varphi_{s,T}(z))-G(\varphi_{s,T}(z),T) \, .
\end{array}
\end{equation}
On the other hand, by differentiating the evolution equation (\ref{eq:L2}) with
respect to $z$, it is easy to see that
$$ t\mapsto d\left(\varphi_{s,t}\right)_z \cdot \left[d\left(\varphi_{s,T}\right)_z\right]^{-1} \cdot \big[h(\varphi_{s,T}(z))-G(\varphi_{s,T}(z),T)\big]$$
is also a solution to (\ref{eq:var2}). By uniqueness, we deduce that for
every $t \ge T$
\begin{eqnarray*}
 \frac{\partial^+
  \varphi^{\eps}_{s,t}(z)}{\partial \eps} \bigg|_{\eps=0}
%&=&\lim
%\limits_{\eps \to 0+}
%\frac{\varphi^{\eps}_{s,T}(z)-\varphi_{s,T}(z)}{\eps} \\
%&=&
= d\left(\varphi_{s,t}\right)_z
\cdot \left[d\left(\varphi_{s,T}\right)_z\right]^{-1} \cdot
\big[ h(\varphi_{s,T}(z))-G(\varphi_{s,T}(z),T)\big] \,.
\end{eqnarray*}
We have hence shown that for fixed $t \ge T$,
$$ \varphi^{\eps}_{s,t}(z)=\varphi_{s,t}(z)+\eps \, d(\varphi_{s,t})_z \cdot \left[
  d(\varphi_{s,T})_z \right]^{-1} \cdot \big[ h(\varphi_{s,T}(z))-G(\varphi_{s,T}(z),T)
\big] +o^{\eps}_{s,t}(z) \, , $$
where   
$$ \lim \limits_{\eps \to 0+} \frac{o^{\eps}_{s,t}(z)}{\eps}=0 \qquad
 \text{ for every }z \in \B^n \, . $$
It is not difficult to prove that this limit actually exists locally uniformly w.r.t.~$z \in
\B^n$. In fact, note that  for fixed $0<r<1$, 
$\varphi^{\eps}_{s,t}(z) \in \overline{\B}^n_r$ for every $z \in
\overline{\B}^n_r$ and all $0 \le s \le t$. Again using the compactness of $\M$,
we see that there is a constant $L_r>0$ such that $||g(z)-g(z')|| \le L_r ||z-z'||$ for
all $g \in \M$ and all $z,z' \in  \overline{\B}^n_r$. Therefore,
\begin{eqnarray*}
|| \varphi^{\eps}_{s,t}(z)-\varphi_{s,t}(z)|| & = & \left|\left|
\varphi^{\eps}_{s,T}(z)-\varphi_{s,T}(z)+ \int \limits_{T}^t
\left( G(\varphi^{\eps}_{s,\tau}(z),\tau)-G(\varphi_{s,\tau}(z),\tau) \right) \, d\tau
\right|\right| \\
& \le & ||\varphi^{\eps}_{s,T}(z)-\varphi_{s,T}(z)||+L_r 
\int \limits_{T}^t || \varphi^{\eps}_{s,\tau}(z)-\varphi_{s,\tau}(z)|| \,
d\tau \, .
\end{eqnarray*}
Now, Gronwall's lemma implies that
\begin{eqnarray*}
 \left|\left|
     \frac{\varphi^{\eps}_{s,t}(z)-\varphi_{s,t}(z)}{\eps}\right|\right| &\le &
\left|\left|\frac{\varphi^{\eps}_{s,T}(z)-\varphi_{s,T}(z)}{\eps}\right|\right|
e^{L_r(t-T)} \\[2mm]
&=& \left|\left| h(\varphi_{s,T}(z))-G(\varphi_{s,T}(z),T) +\frac{r_{\eps}(z)}{\eps}\right|\right|
\, e^{L_r(t-T)} \, .
\end{eqnarray*}
Hence, $o^{\eps}_{s,t}(z)/\eps$ is uniformly bounded on $\overline{\B}^n_r$
as $\eps \to 0+$. Since we have already proved that $o^{\eps}_{s,t}(z)/\eps
\to 0$ pointwise in $\B^n$, Vitali's theorem  shows that actually
$o^{\eps}_{s,t}/\eps\to 0$ locally uniformly in $\B^n$.
\end{proof}

\hide{\begin{remark}
The above proof works for every $s \ge 0$ and every $T>s$ such that the ``left
derivative''
$$ \frac{\partial^- \varphi_{s,t}(z)}{\partial t}\bigg|_{t=T}:=\lim
\limits_{\eps \to 0+} \frac{\varphi_{s,T}(z)-\varphi_{s,T-\eps}(z)}{\eps}$$
exists locally uniformly w.r.t.~$z \in \B^n$.
\end{remark}}

\hide{

\begin{proposition}
Let $G(z,t)$ be a Herglotz vector field in the class $\M$,
let $(\varphi_{s,t}):=(\varphi_{s,t}^G)$  and let $T \ge 0$ be fixed. 
Let
$$ G^{-}_{\eps}(\cdot,t):=\begin{cases}
G(\cdot,t) & \text{if } t \in \R^+_0 \backslash (T-\eps,T) \\
G(\cdot,T) & \text{if } t \in (T-\eps,T) \, ,
\end{cases} $$

$$
 G^{+}_{\eps}(\cdot,t):=\begin{cases}
G(\cdot,t) & \text{if } t \in \R^+_0 \backslash (T,T+\eps) \\
G(\cdot,T) & \text{if } t \in (T,T+\eps)
\end{cases}
$$
and let $\varphi^{\eps,-}_{s,t}:=\varphi_{s,t}^{G^-_{\eps}}$,
$\varphi^{\eps,+}_{s,t}:=\varphi_{s,t}^{G^+_{\eps}}$. If $s<T$, then 
$$ \varphi^{\eps,\pm}_{s,T\pm \eps}(z)=\varphi_{s,T}(z) \pm \eps
G(\varphi_{s,T}(z),T)+r_{\eps}(z) \, ,$$
Here $r_{\eps}/\eps \to 0$ locally uniformly in $\B^n$ as $\eps \to 0+$.
\end{proposition}

\begin{proof} \textbf{To do}
\end{proof}}

\section{Variations in $\mathcal{S}_n^0$ and variations for parametric
  representations} \label{sec:3}

By definition,  every $f \in \mathcal{S}_n^0$ has the form $$f=\lim \limits_{t \to \infty} e^t\varphi^G_{0,t}$$
for some Herglotz vector field $G(z,t)$ in the class $\M$. Therefore
the following result for $s=0$ and $t=\infty$ is exactly the statement of
Theorem \ref{thm:main1} and provides us with a variational formula in the class $\mathcal{S}_n^0$.

\begin{theorem} \label{thm:var2}
Let $G(z,t)$ be a Herglotz vector field in the class $\M$, let $T \in R_G$ and $h \in \M$.
For fixed $s \in [0,T]$ consider the needle variations
$(\varphi^{\eps}_{s,t})$ of $(\varphi_{s,t}):=(\varphi^G_{s,t})$ with data $(T,h)$.
Then
\begin{equation} \label{eq:varI}
 e^t\varphi^{\eps}_{s,t}=e^t\varphi_{s,t}+\eps \, d(e^t\varphi_{s,t})_z
\cdot \left[
  d(e^T\varphi_{s,T})_z \right]^{-1} \cdot e^T \big[ h(\varphi_{s,T})-G(\varphi_{s,T},T)
\big] +r^{\eps}_{s,t} \,  
\end{equation}
for any $t \in [T,\infty]$. Here, 
 the error term $r^{\eps}_{s,t} \in \HB$ has the property that
$r^{\eps}_{s,t}/\eps \to 0$  locally uniformly in $\B^n$ for every
fixed $t \in [T,\infty]$ as $\eps \to 0+$. 
\end{theorem}

\begin{remark} 
Note that Theorem \ref{thm:var2} holds in particular for $t=\infty$, where we have used the
convenient notation $$e^t \varphi^{\eps}_{s,t}:=\lim_{\tau \to \infty} e^{\tau}
\varphi^{\eps}_{s,\tau}\quad \text{ for } t=\infty \, .$$
In this case, we define the error term $r^{\eps}_{s,\infty}\in \HB$ as 
$$ r^{\eps}_{s,\infty}:=\lim \limits_{t \to \infty} r^{\eps}_{s,t} \, .$$
This limit clearly exists locally uniformly in $\B^n$ in view of (\ref{eq:varI}).
\end{remark}

\begin{proof}[Proof of Theorem \ref{thm:var2}]
Let $r^{\eps}_{s,t}$ be defined by (\ref{eq:varI}). We need to show that
$r^{\eps}_{s,t}/\eps \to 0$ locally uniformly in $\B^n$ for
every fixed $t \in [T,\infty]$  as $\eps \to 0+$.
The cases  $t<\infty$  follow directly from Theorem
\ref{thm:varpar}, so we only need to deal with the case $t=\infty$.

\medskip

(i) \, In order to handle the error term $r^{\eps}_{s,\infty}$  we first derive a convenient expression
for the error term $ r^{\eps}_{s,t}$ for all $0 \le s \le t<\infty$.
Let $v_{s,t}^{\eps}(z):=e^t\varphi^{\eps}_{s,t}(z)$ and $v^0_{s,t}(z):=e^t\varphi_{s,t}(z)$.
If we set $\tilde{G}(z,t):=z+e^t G(e^{-t} z,t)$, then 
\begin{equation} \label{eq:var10} \begin{array}{rcl}
 \dot{v}^{\eps}_{s,t}(z) &=& \tilde{G}(v^{\eps}_{s,t}(z),t) \, , \qquad t \ge T \, \\[2mm]
v^{\eps}_{s,T}(z)&=&v^0_{s,T}(z)+\eps e^T \big[h(\varphi_{s,T}(z))-G(\varphi_{s,T}(z),T)\big]+r^{\eps}_{s,T}(z)
\, ,
\end{array}
\end{equation}
where $r^{\eps}_{s,T}/\eps \to 0$  locally uniformly in
$\B^n$ as $\eps \to 0+$ by applying Theorem \ref{thm:varpar} for $t=T$.
For $t \ge T$ let 
$$ E^{\eps}_{s,t}(z):=\int \limits_{0}^1 \frac{\partial \tilde{G}}{\partial
    z}\left(v^0_{s,t}(z)+\alpha \left( v^{\eps}_{s,t}(z)-v^{0}_{s,t}(z) \right),t\right)
  \, d\alpha \, , \quad E^0_{s,t}(z)=\frac{\partial \tilde{G}}{\partial z}
  \left(v^0_{s,t}(z),t \right) \, ,$$
so the difference $\Psi^{\eps}_{s,t}(z):=v^{\eps}_{s,t}(z)-v^0_{s,t}(z)$ has the property
 \begin{equation} \label{eq:var11} \begin{array}{rcl}
 \dot{\Psi}^{\eps}_{s,t}(z) &=& E^{\eps}_{s,t}(z) \cdot \Psi^{\eps}_{s,t}(z) \, ,
 \qquad t \ge T \\[2mm]
\Psi^{\eps}_{s,T}(z)&=& \eps e^T \big[h(\varphi_{s,T}(z))-G(\varphi_{s,T}(z),T)\big]+r^{\eps}_{s,T}(z)
\, .
\end{array}
\end{equation}
In order to analyze the behaviour of $\Psi^{\eps}_{s,t}$ as $t \to \infty$, we
consider the linear matrix--ODE
\begin{equation} \label{eq:var12a} \begin{array}{rcl}
 \dot{Y}^{\eps}_{s,t}(z) &=&  E^{\eps}_{s,t}(z) \cdot Y^{\eps}_{s,t}(z) \, ,
 \qquad t \ge T \\[2mm]
Y^{\eps}_{s,T}(z)&=& I 
\, .
\end{array}
\end{equation}
The motivation for doing so comes from the observation that
 in view of (\ref{eq:var11}) we can write 
\begin{equation} \label{eq:varII}
\begin{array}{rcl}
 \Psi^{\eps}_{s,t}(z)&=&Y^{\eps}_{s,t}(z) \cdot \Psi^{\eps}_{s,T}(z) \\[2mm] &=& 
 Y^{\eps}_{s,t}(z) \cdot \left\{\eps e^T
   \big[h(\varphi_{s,T}(z))-G(\varphi_{s,T}(z),T)\big]+r^{\eps}_{s,T}(z)\right\} \, .
\end{array}
\end{equation}
In a similar way, since differentiating (\ref{eq:var10}) for $\eps=0$
w.r.t.~$z$ shows that
$$ \frac{d}{dt} \left[ d(v^0_{s,t})_z \right]=\frac{\partial
  \tilde{G}}{\partial z}\left( v^0_{s,t}(z),t \right) \cdot \left[
  d(v^0_{s,t})_z \right]=E^0_{s,t}(z) \cdot \left[
  d(v^0_{s,t})_z \right] \, , \qquad t \ge T \, ,$$
we get 
\begin{equation} \label{eq:varIII}
 d(v^0_{s,t})_z=Y^0_{s,t}(z) \cdot d (v^0_{s,T})_z \, , \qquad t \ge T \,
. \end{equation}
Now, formulas (\ref{eq:varII}) and (\ref{eq:varIII}) and 
the definition of the error term $r^{\eps}_{s,t}$ show that
\begin{equation} \label{eq:var12}
r^{\eps}_{s,t}(z)=
\eps e^T \left[ Y^{\eps}_{s,t}(z)-Y^0_{s,t}(z) \right] \cdot
\big[h(\varphi_{s,T}(z))-G(\varphi_{s,T}(z),T)\big]+Y^{\eps}_{s,t}(z)
r^{\eps}_{s,T}(z) \, .
\end{equation}
(ii) \, We now examine $Y^{\eps}_{s,t}$ with the help 
of the linear matrix--ODE (\ref{eq:var12a}) and show that for any $r \in (0,1)$
there exists a constant $M_r>0$ such that
\begin{equation} \label{eq:var13}
 || Y^{\eps}_{s,t}(z)-Y^0_{s,t}(z)|| \le M_r \eps \, \text{ and } \,
||Y^{\eps}_{s,t}(z)|| \le M_r \, \text{ for all } \, ||z|| \le r \, \text{ and
  all } \, t \ge
  T\, . \end{equation}
In view of (\ref{eq:var12}) this then implies $r^{\eps}_{s,\infty}(z)/\eps \to
0$ uniformly in $||z|| \le r$ as $\eps \to 0+$.
It therefore remains to prove (\ref{eq:var13}). We fix $r \in (0,1)$.
In the following, $C_r$ always denotes a constant, which depends only on $r$,
but the value of $C_r$ may be different at each occurence.
We first note 
$\left|\left|\id+d(h)_z\right|\right| \le C_r \cdot ||z||$  for all $||z|| \le r$ and all  $h \in \M$.
This follows from the compactness of $\M$ and the normalization $d(h)_0=-\id$.
Moreover, $||\varphi_{s,\tau}(z)|| \le C_r e^{-\tau}$ for all $\tau \ge T$ and
all $||z|| \le r$, see \cite[Lemma 8.1.4]{GK}.
Hence, from the definition of $E^{\eps}_{s,\tau}(z)$ we infer that
\begin{equation} \label{eq:var14}
\begin{array}{rcl}
||E^{\eps}_{s,\tau}(z)|| &\le& \displaystyle \int \limits_{0}^{1}
\left|\left|\,\id+\displaystyle\frac{\partial G}{\partial z}\left(\varphi_{s,\tau}(z)+\alpha \left(
      \varphi^{\eps}_{s,\tau}(z)-\varphi_{s,\tau}(z) \right) ,\tau\right)\right|\right|\,
d\alpha \\ 
 &\le & \displaystyle C_r \int \limits_{0}^{1} \left|\left| \varphi_{s,\tau}(z)+\alpha \left(
      \varphi^{\eps}_{s,\tau}(z)-\varphi_{s,\tau}(z) \right) \right|\right|\,d \alpha\\
&\le & C_r \, e^{-\tau} \, \text{ for all } ||z|| \le r \text{ and all } \tau \ge T
\,.
\end{array}
\end{equation}
In a similar way, we can deduce 
\begin{equation}
 \label{eq:var15}
\begin{array}{rcl}
||E^{\eps}_{s,\tau}(z)-E^0_{s,\tau}(z)|| &\le&  C_r
||\varphi^{\eps}_{s,\tau}(z)-\varphi_{s,\tau}(z)||\\[2mm] &=&  C_r
||\varphi_{T,\tau}(\varphi^{\eps}_{s,T}(z))-\varphi_{T,\tau}(\varphi_{s,T}(z))||\\[2mm]
& \le & C_r e^{-\tau} ||\varphi^{\eps}_{s,T}(z)-\varphi_{s,T}(z)||\\[2mm]
& \le & C_r  e^{-\tau} \eps 
 \, \text{ for all } ||z|| \le r \text{ and all } \tau \ge T
\, .
\end{array}
\end{equation}
The last estimate comes from (\ref{eq:6}).
We are now prepared to prove (\ref{eq:var13}).
Since $t \mapsto Y^{\eps}_{s,t}(z)$ is a solution to (\ref{eq:var12a}), we get
\begin{eqnarray*}
 ||Y^{\eps}_{s,t}(z)-Y^0_{s,t}(z)|| 
&=& \left| \left| \int \limits_{T}^t E^{\eps}_{s,\tau}(z) Y^{\eps}_{s,\tau}(z)
    \, d\tau-\int \limits_{T}^{\tau} E^0_{s,\tau}(z) V^0_{s,\tau}(z) \, d\tau
  \right| \right| \\
&\le&   \int \limits_{T}^t
||E^{\eps}_{s,\tau}(z)|| \cdot ||Y^{\eps}_{s,\tau}(z)-Y^0_{s,\tau}(z)|| \,
d\tau\\ & & +\int \limits_{T}^t ||E^{\eps}_{s,\tau}(z)-E^0_{s,\tau}(z)|| \,
||Y^0_{s,\tau}(z)|| \, d\tau \, . 
\end{eqnarray*}

Now (\ref{eq:varIII}) shows 
$$Y^0_{s,\tau}(z)=e^{\tau} d\left(\varphi_{s,\tau}\right)_z
\cdot \left[ d\left(e^T\varphi_{s,T}\right)_z \right]^{-1} \, ,$$
so the inequality $||\varphi_{s,\tau}(z)|| \le C_r e^{-\tau}$ for all $||z|| \le
r$ and all $\tau \ge T$, which leads to $||d(\varphi_{s,\tau})_z|| \le C_re^{-\tau}$, 
therefore implies that
 $||Y^0_{s,\tau}(z)|| \le C_r$ for all
$\tau \ge T$.  Hence, in combination with (\ref{eq:var14}) and (\ref{eq:var15}), 
we get the implicit estimate
\begin{eqnarray*}
 ||Y^{\eps}_{s,t}(z)-Y^0_{s,t}(z)|| \le   C_r \int \limits_{T}^t
e^{-\tau} ||Y^{\eps}_{s,\tau}(z)-Y^0_{s,\tau}(z)|| \,
d\tau +C_r \eps \int \limits_{T}^t e^{-\tau} \, d\tau \, , 
\end{eqnarray*}
which is valid for all $||z|| \le r$ and all $t \ge T$.
Again using Gronwall's lemma, we obtain
$|| Y^{\eps}_{s,t}(z)-Y^0_{s,t}(z)|| \le C_r \eps$ 
and then also $|| Y^{\eps}_{s,t}(z)|| \le ||
Y^{\eps}_{s,t}(z)-Y^0_{s,t}(z)||+||Y^0_{s,t}(z)||
\le C_r$ for all $||z|| \le r$ and
all $t \ge T$.
 This proves (\ref{eq:var13}) and finishes the proof of
Theorem \ref{thm:var2} for the case $t=\infty$.
\end{proof}

Theorem \ref{thm:var2} enables us to construct variations for a certain class 
of Loewner chains. We first recall the basic concepts.

\begin{definition}
A \textit{normalized Loewner chain} $(f_t)_{t \ge 0}$ is a family of univalent
functions $f_t : \B^n \to \C^n$ such that $f_t(0)=0$, $d(f_t)_0=e^t \id$ for
all $t \ge 0$ and such that for every $0 \le s \le t$ there exists a holomorphic map
$\varphi_{s,t} : \B^n \to \B^n$ with $f_s=f_t \circ \varphi_{s,t}$.
A normalized Loewner chain $(f_t)_{t \ge 0}$ is called a \textit{parametric
  representation} if the family $\{e^{-t} f_t\}$ is normal. 
\end{definition}

\begin{remark} \label{rem:1}
We note the following well--known facts, see \cite{GK}.
\begin{itemize}
\item[(a)]
If $(f_t)$ is a normalized Loewner chain, then there is a unique
Herglotz vector field $G(z,t)$ in the class $\M$ such that
the Loewner--Kufarev PDE
\begin{equation} \label{eq:l3}
\frac{\partial f_t}{\partial t}(z)=- d(f_t)_z   \cdot G(z,t)
\end{equation}
holds.
\item[(b)] If  $G(z,t)$ is a Herglotz vector field in the class $\M$, we define
$$f^G_s:=\lim \limits_{t \to \infty} e^t \varphi^G_{s,t}\, .$$
Then  $(f^G_t)_{t \ge 0}$ is a parametric
representation. In fact, $(f^G_t)_{t \ge 0}$ is the unique parametric
representation such that $f_t^G$ is a solution to
(\ref{eq:l3}) for the Herglotz vector field $G(z,t)$. 
The Loewner chain $(f^G_t)$ is called the \textit{canonical solution of the Loewner PDE (\ref{eq:l3})}.
\item[(c)] It follows from part (b) that the class $\mathcal{S}_n^0$ consists
  precisely of all normalized univalent functions $f \in \HB$ for which
there is  a parametric representation $(f_t)_{t \ge 0}$ with $f_0=f$. 
\end{itemize}
\end{remark}

We now construct for a given parametric representation $(f_t)_{t \ge 0}$
a differentiable family of deformations $(f^{\eps})_{t \ge 0}$ that coincide
with $(f_t)_{t \ge 0}$ from a certain time on.

\begin{theorem}[Variations of parametric representations] \label{thm:parvar}
Let $(f_t)_{t \ge 0}$ be a parametric representation  with associated Herglotz
vector field $G(z,t)$ in the class $\mathcal{M}_n$. Let
$(\varphi_{s,t})_{0 \le s \le t}$ denote the evolution family generated by $G(z,t)$.
Then for any $T \in R_G$, any $h \in \M$ and any $\eps \in
(0,T)$ there exists a parametric representation $(f^{\eps}_t)_{t \ge 0}$ such that
$$ f^{\eps}_t=\begin{cases}
f_t & \text{ if } t  \ge T\, , \\
f_t+\eps \, d(f_t)_z \cdot \left[ d(e^T\varphi_{t,T})_z \right]^{-1} \cdot e^T
\left[
  h(\varphi_{t,T})-G(\varphi_{t,T},T)\right]+o_t^{\eps} & \text{ if } t < T
\, .
\end{cases}$$
Here,  $o^{\eps}_{t}$ indicates a term, which divided by $\eps$, tends to $0$
 locally uniformly in $\B^n$ as $\eps\to 0+$. 
\end{theorem}

\begin{proof}[Proof]
Let $G(z,t)$ denote the Herglotz vector field in the class $\M$
such that the Loewner PDE (\ref{eq:l3}) holds, so $f_t=\lim_{\tau \to
  \infty} e^{\tau} \varphi_{t,\tau}$ for any $t \ge 0$ in view of Remark
\ref{rem:1} (b).
Denote by $\varphi^{\eps}_{s,t}$ the needle variations of $\varphi_{s,t}$ with
data $(T,h)$. Define $f^{\eps}_t=\lim_{\tau \to
  \infty} e^{\tau} \varphi^{\eps}_{t,\tau}$ for any $t \ge 0$.
Since $\varphi^{\eps}_{t,\tau}=\varphi_{t,\tau}$ for $T \le t \le \tau$, we
have $f^{\eps}_t=f_t$ for any $t \ge T$. Now let $t<T$ and choose $\tau \ge
T$. Then Theorem \ref{thm:var2} shows that
$$ e^{\tau} \varphi^{\eps}_{t,\tau}=e^{\tau} \varphi_{t,\tau}+\eps \, d(e^{\tau}\varphi_{t,\tau})_z \cdot
\left[ d(e^{T} \varphi_{t,T})_z \right]^{-1} \cdot e^T \big[
h(\varphi_{t,T})-G(\varphi_{t,T},T) \big]+o^{\eps}_{t,\tau} \, .$$
Here, $o^{\eps}_{t,\tau}/\eps \to 0$ as $\eps\to 0+$ locally uniformly in
$\B^n$. The proof is finished by letting $\tau=\infty$.
\end{proof}

As we have already pointed out in the introduction, Theorem \ref{thm:parvar} is related to the recent work \cite{BHKG}, where
variations of a specific class of Loewner chains (so--called ger\"aumig
Loewner chains) have been introduced, see in particular Theorem 3.1 in~\cite{BHKG}.

\section{Extremal problems on $\mathcal{S}_n^0$} \label{sec:4}

In order to apply our variational formulas, we need to consider a suitable class
of ``differentiable nonlinear functionals'' on the Fr\'echet space $\HB$. We
use the Fr\'echet space calculus as developed by R.~Hamilton
\cite{Ham82}. This approach is more general than the one used in the standard
monographs \cite{Pom} or \cite{Dur83}.

\begin{definition}[Complex derivative, \cite{Ham82}, p.~73] \label{def:der}
Let $U \subseteq \HB$ be an open set and $\Phi : U \to \C$ continuous.
We call $\Phi : U \to \C$ \textit{differentiable at $f \in U$ along $h \in \HB$}, if
the limit
$$ \Lambda(f;h):=\lim \limits_{\C \ni \delta \to 0} \frac{\Phi(f+\delta
  h)-\Phi(f)}{\delta}$$
exists. In this case, $\Lambda(f;h)$ is called the \textit{directional derivative of $\Phi$
at $f$ along $h$}.
We say that $\Phi : U \to \C$ is \textit{complex differentiable at $F \in U$}, if 
there is an open neighborhood $V \subseteq U$ of $F$ such that $\Phi$ is differentiable at
any $f \in V$ along any $h \in \HB$ and if the map $\Lambda : V \times \HB \to \C$
is continuous. In this case, $L:=\Lambda(F,\cdot)$ is called the
\textit{complex derivative of $\Phi$ at $F$}. 
\end{definition}
 
\begin{lemma} \label{lem:ham}
Let $U \subseteq \HB$ be an open set and let  $\Phi : U \to \C$ be complex
differentiable at $F \in U$ with complex derivative $L=\Lambda(F; \cdot)$. Then
\begin{itemize}
\item[(a)] The continuous functional $L : \HB \to \C$ is linear.
\item[(b)] If  $h \in \HB$ and $f^{\eps}=F+\eps h+r_{\eps}$, where  $r_{\eps}/\eps \to 0$
 locally uniformly in $\B^n$ as $\eps \to0+$, then
$$ \lim \limits_{\eps \to 0+}  \frac{\Phi(f^{\eps})-\Phi(F)}{\eps}=L(h)\, .$$
\end{itemize}
\end{lemma}

\begin{proof} (a) See \cite[p.~76--77]{Ham82}.\\
(b) Let $U \subseteq \HB$ be an open neighborhood of $F$ such that $$\Lambda : U
\times \HB \to \C$$ is continuous. We may assume that $U$ is convex. Lemma
3.3.1 in \cite{Ham82} shows that there is a continuous mapping $\hat{L} : U \times U
\times \HB \to \C$ so that $h \mapsto \hat{L}(f_1,f_2,h)$ is linear and such that
$\Phi(f_2)-\Phi(f_1)=\hat{L}(f_1,f_2,f_2-f_1)$ for all $f_1,f_2 \in U$. In addition,
$\hat{L}(f_1,f_1,h)=\Lambda(f_1,h)$ for every $f_1 \in U$ and every $h \in \HB$.
Therefore,
$$
\frac{\Phi(f^{\eps})-\Phi(F)}{\eps}=\frac{\hat{L}(F,f^{\eps},f^{\eps}-F)}{\eps}=
\frac{\hat{L}(f^{\eps},F,\eps h+r_{\eps})}{\eps} \to \hat{L}(F,F,h)=L(h)$$
as $\eps\to0+$.
\end{proof}

\begin{corollary} \label{cor:diff}
Let $G(z,t)$ be a Herglotz vector field in the class $\M$, let $T \in R_G$ and $h \in \M$.
Consider the needle variations
$(\varphi^{\eps}_{t})$ of $(\varphi_{t}):=(\varphi^G_{t})$ with data $(T,h)$.
Suppose that $\Phi$ is a complex functional with complex derivative $L$ at
$e^{\tau}\varphi_{\tau}$ for some fixed $\tau \in (T,\infty]$.
Then
$$ \lim \limits_{\eps \to 0+} \frac{\Phi(e^{\tau} \varphi^{\eps}_{\tau})-\Phi(e^{\tau}
  \varphi_{\tau})}{\eps}=
L\left(  d(e^{\tau}\varphi_{\tau})_z \cdot \left[
  d(e^T\varphi_{T})_z \right]^{-1} \cdot e^T \big[ h(\varphi_{T})-G(\varphi_{T},T)
\big] \right) \, . $$
\end{corollary}

\begin{proof}
This follows from Theorem \ref{thm:var2} for $s=0$ and Lemma \ref{lem:ham} (b).
\end{proof}

\begin{theorem} \label{thm:main2}
Let $\Phi$ be a complex functional with complex derivative $L$ at $F \in
\mathcal{S}_n^0$ and suppose that $F$ maximizes $\Re \Phi$ over $\mathcal{S}_n^0$.
Let $G(z,t)$ be a Herglotz vector field in the class $\M$ with $(\varphi_t):=(\varphi^G_t)$
 such that $F=e^{\tau} \varphi_{\tau}$ for some $\tau \in (0,\infty]$.
Then the fowllowing hold.
 \begin{enumerate}
\item[(a)] For every $t \in (0,\tau) \cap R_G$,
 the function $G(\cdot,t)\in \mathcal{M}_n$ maximizes the real part of the continuous linear
  functional $$L_t(h):=L\left(d(F)_z \cdot   \left[
  d(\varphi_t)_z\right]^{-1} \cdot h(\varphi_t) \right)$$ over $\mathcal{M}_n$,
that is,
 $$ \max \limits_{h \in \mathcal{M}_n} \Re L_t(h) = \Re L_t(G(\cdot,t)) \, .$$
\hide{
$$ \max \limits_{h \in \mathcal{M}_n} \Re L\left((dF)_z\cdot  \left[ d(e^t
  \varphi_t)_z \right]^{-1} \cdot h(\varphi_t) \right)=\Re L \left((dF)_z
\left( d(e^t
  \varphi_t)_z \right]^{-1} \cdot G(\varphi_t,t) \right)\, . $$}
\item[(b)] The function $t \mapsto \max \limits_{h \in \mathcal{M}_n} \Re L_t(h)$
 is constant on $[0,\tau)$.
\end{enumerate}
\end{theorem}

Theorem \ref{thm:main2a} is the special case $\tau=\infty$ of Theorem
\ref{thm:main2} (a).
 Theorem \ref{thm:main2} (a) for $\tau<\infty$ and $n=1$ is exactly
  Theorem 4.1 in \cite{Ro}.

\begin{proof}[Proof of Theorem \ref{thm:main2}]
(a) Let $h \in \M$, $T \in (0,\tau)\cap R_G$ and let $(\varphi_t^{\eps})$ denote the needle variations of $(\varphi_t)$ 
with data $(T,h)$. Since $F$ maximizes $\Re \Phi$ on $\mathcal{S}_n^0$, 
we have $\Re \Phi(F) \ge \Re \Phi(e^{\tau} \varphi^{\eps}_{\tau})$
for every $\eps>0$. Corollary \ref{cor:diff} therefore implies
$$\Re L\left(  d(F)_z \cdot \left[
  d(e^T\varphi_{T})_z \right]^{-1} \cdot \big[ h(\varphi_{T})-G(\varphi_{T},T)
\big] \right) \le 0 \, .$$

\medskip

(b) Let $H(t,h):=L_t(h)=L \left(d(F)_z  \cdot \left[d
  (\varphi_t)_z\right]^{-1} \cdot h(\varphi_t) \right)$ and
$$ m(t):=\max \limits_{h \in \M} \Re H(t,h) \, .$$
In order to show that $m$ is constant on $[0,\tau)$ we proceed in several
steps.

\smallskip

(b1) We first show that $m : [0,\tau) \to \R$ is locally Lipschitz continuous. Since $L$ is a continuous
linear functional on the Fr\'echet space $\HB$, there are finite complex
Borel measures $\mu_1,\ldots, \mu_n$ which are supported on  compact subsets
$E_1,\ldots, E_n$ of $\B^n$
such that
\begin{equation} \label{eq:fd0}
 L(h)=\sum \limits_{k=1}^n \iint \limits_{E_k} h_k(z) \, d\mu_k(z) \, , \qquad
 h=(h_1,\ldots, h_n) \in \HB \, ,
\end{equation}
see e.g.~\cite[p.~65]{GKP}.
Let $E$ be a closed ball in $\B^n$ centered at
the origin such that $E_k \subset E$ for $k=1,\ldots, n$.
Since $||\varphi_t(z)|| \le ||z||$ for every $z \in \B^n$ and every $t \ge 0$, it
follows that
\begin{equation} \label{eq:fd1}
\varphi_t(z) \in E\text{ for all } z \in E \text{ and all } t \ge 0 \, .
\end{equation}
As $\M$ is a compact subset of $\HB$, we see as before that there exists a constant $\gamma>0$ such that
\begin{equation} \label{eq:fd2}
|| h(z)|| \le \gamma ||z|| \, \text{ and } \, ||h(z)-h(z')|| \le \gamma ||z-z'||\quad \text{ for all
} z,z' \in E , \,  h \in \M \, ,
\end{equation}
see e.g.~formula (8.1.2) in \cite{GK}.
Since $\varphi_t(z)$ is a solution to (\ref{eq:L}) and $G(\cdot,t) \in  \M$
for a.e.~$t \ge 0$, the estimate
(\ref{eq:fd2}) combined with (\ref{eq:fd1}) implies
\begin{equation} \label{eq:fd3}
|| \varphi_{\beta}(z)-\varphi_{\alpha}(z) ||= \left|\left| \int
    \limits_{\alpha}^{\beta} G(\varphi_t(z),t) \, dt \right|\right| \le \gamma
\cdot |\beta-\alpha|
\quad \text{ for all } \alpha,\beta \ge 0 \, , z \in E \, .
\end{equation}

In a similar way, since $t \mapsto \left[ d(\varphi_t)_z \right]^{-1}$ has the
property that
$$ \frac{d}{dt} \left(  \left[ d\left(\varphi_t\right)_z \right]^{-1} \right)=-  \left[
  d\left(\varphi_t\right)_z \right]^{-1} \cdot \frac{\partial G}{\partial
  z}(\varphi_t(z),t) \quad \text{ for a.e. } t \ge 0 \, , $$
an application of Gronwall's lemma leads to
\begin{equation} \label{eq:fd5}
 \left|\left| \left[ d \left(\varphi_{\beta} \right)_z \right]^{-1} \right|
\right| \le e^{\gamma \beta} %1+\gamma \beta 
\quad \text{ for
  all } \beta \ge 0 \, , z \in E \, ,
\end{equation}

and

\begin{equation} \label{eq:fd4}
\left|\left|  \left[ d\left(\varphi_{\beta}\right)_z \right]^{-1}- \left[ d\left(\varphi_{\alpha}\right)_z
    \right]^{-1} \right|\right| \le \left|e^{\gamma \beta}-e^{\gamma \alpha}\right| %\gamma
                                %\cdot |\beta-\alpha| \quad 
\text{ for all } \alpha,\beta \ge 0 \, , z \in E \, .
\end{equation}

We can now prove that $m : [0,\tau) \to \R$ is locally  Lipschitz continuous.
Let $ \alpha,\beta \ge 0$ be given. Since $\M$ is a compact subset of
$\HB$, there is a function $h_{\beta} \in \M$ such that $m(\beta)=\Re H(\beta,h_{\beta})$.
In view of $m(\alpha) \ge \Re H(\alpha,h_{\beta})$ it follows that
\begin{eqnarray*}
m(\beta)-m(\alpha) & \le & \Re H(\beta,h_{\beta})-\Re H(\alpha,h_{\beta})\\ &&
\hspace*{-2cm} = 
\Re L \left( d\left(F\right)_z \cdot \left[  d\left(\varphi_{\beta} \right)_z
  \right]^{-1} \cdot h_{\beta}(\varphi_{\beta}) \right) 
- \Re L \left( d\left(F\right)_z \cdot \left[  d\left(\varphi_{\alpha} \right)_z
  \right]^{-1} \cdot h_{\beta}(\varphi_{\alpha}) \right) \\
&& \hspace*{-2cm} = \Re L \left( d\left(F\right)_z \cdot \left[ d\left(\varphi_{\beta} \right)_z
  \right]^{-1} \cdot \left(
    h_{\beta}(\varphi_{\beta})-h_{\beta}(\varphi_{\alpha}) \right) \right)\\
& & +  \Re L \left( d\left(F\right)_z \cdot \left( \left[  d\left(\varphi_{\beta} \right)_z
  \right]^{-1}-\left[ d\left(\varphi_{\alpha} \right)_z \right]^{-1} \right)
\cdot h_{\beta}(\varphi_{\alpha}) \right) \\
& & \hspace*{-2cm} = \Re \sum \limits_{k=1}^n \iint \limits_{E_k} \left(
  d(F)_z \cdot \left[ d\left(\varphi_{\beta}
  \right)_z \right]^{-1} \cdot  \left(
    h_{\beta}(\varphi_{\beta}(z))-h_{\beta}(\varphi_{\alpha}(z)) \right) \right)_k \,
  d\mu_k(z)\\
& & \hspace*{-1cm} + \Re \sum \limits_{k=1}^n \iint \limits_{E} \left( d\left(F\right)_z \cdot \left( \left[ d\left(\varphi_{\beta} \right)_z
  \right]^{-1}-\left[ d\left(\varphi_{\alpha} \right)_z \right]^{-1} \right)
\cdot h_{\beta}(\varphi_{\alpha}(z)) \right)_k \, d\mu_k(z)
\, ,
\end{eqnarray*}
where we have used the representation formula (\ref{eq:fd0}).
In view of the estimates
(\ref{eq:fd1})--(\ref{eq:fd4}), we now see that for every compact subintervall $I$
of $[0,\tau)$ 
there is a constant $C=C_I$ 
 such that
$m(\beta)-m(\alpha) \le C |\beta-\alpha|$ for all $\alpha,\beta \in I$.  This shows that $m : [0,\tau) \to
\R$ is locally Lipschitz.

\smallskip

(b2) We next show that
$$ \frac{\partial H}{\partial t}(t,G_t)=0 \quad \text{ for a.e.} \quad t \ge 0
\, .$$
As above, using the fact that $\varphi_t$ is a solution of the Loewner equation (\ref{eq:L}),
 we first see that there is a set $E \subseteq \R^+$ of measure $0$ such that
for any $t \in \R^+\backslash E $
$$ \frac{d}{dt} \left( \left[ d(\varphi_t)_z \right]^{-1} \right)=- \left[ d(\varphi_t)_z
\right]^{-1} \cdot \frac{\partial G}{\partial z}(\varphi_t(z),t) \quad \text{
  locally uniformly w.r.t.~} z \in \B^n\, .$$
Now, for any $t,t^* \in \R^+$, we have
\begin{eqnarray*}
& &  \hspace*{-0.7cm}\frac{ \left[ d(\varphi_t)_z \right]^{-1}
  G(\varphi_t(z),t^*)- \left[ d(\varphi_{t^*})_z \right]^{-1}
  G(\varphi_{t^*}(z),t^*)}{t-t^*}= \\ && = \frac{ \left[ d(\varphi_t)_z \right]^{-1}-\left[ d(\varphi_{t^*})_z
  \right]^{-1}}{t-t^*} G(\varphi_t(z),t^*) +\left[ d(\varphi_{t^*})_z
\right]^{-1} \frac{G(\varphi_t(z),t^*)-G(\varphi_{t^*}(z),t^*)}{t-t^*} \, ,
\end{eqnarray*}
and this expression converges for $t \to t^*
\in \R^+\backslash E$ to
$$  -\left[ d(\varphi_{t^*})_z \right]^{-1} \cdot
\frac{\partial G}{\partial z}(\varphi_{t^*}(z),t^*) G(\varphi_{t^*}(z),t^*)+\left[ d(\varphi_{t^*})_z
\right]^{-1} \frac{\partial G}{\partial z}(\varphi_{t^*}(z),t^*) G(\varphi_{t^*}(z),t^*)
=0 $$
locally uniformly w.r.t.~$z \in \B^n$. Hence, by  definition of $H$, 
\begin{eqnarray*}
 \lim \limits_{t \to t^*}
 \frac{H(t,G(\cdot,t^*))-H(t^*,G(\cdot,t^*))}{t-t^*}=0 \, .
\end{eqnarray*}
for every $t^* \in \R^+\backslash E$.

\smallskip

(b3)  Next note that $m(t)=\Re H(t,G_t)$ for every $t \in (0,\tau) \cap R_G$ by part (a).
Therefore, we obtain for all $t,t^* \in (0,\tau) \cap R_G$ such that $t^*<t$, 
$$ \frac{\Re H(t,G_{t^*})-\Re H(t^*,G_{t^*})}{t-t^*} \le
\frac{m(t)-m(t^*)}{t-t^*} \le \frac{\Re H(t,G_{t})-\Re H(t^*,G_{t})}{t-t^*}
\, .$$
Since $m : [0,\tau) \to \R$ is locally Lipschitz continuous, it is differentiable for
a.e.~$t \ge 0$, so 
$$ \frac{d}{dt} m(t)=\frac{\partial \Re H}{\partial t}(t,G_t) \qquad \text{
  for a.e. }
t \in (0,\tau) \, .$$
By what we have proved in  (b2), we see that 
$$ \frac{d}{dt} m(t)=0 \qquad \text{
  for a.e. }
t \in (0,\tau) \, .$$
Therefore, the locally Lipschitz continuous function
 $m : [0,\tau) \to \R$ is constant on $[0,\tau)$.
\end{proof}

\begin{proof}[Proof of Theorem \ref{thm:pommerenke}]
Since $\M$ is compact and $h(0)=0$, $dh_0=-\id$ for every $h \in \M$,
there exists for every $0<r<1$ a constant $M_r>0$ such that
\begin{equation} \label{eq:p1}
||h(z)+z|| \le M_r ||z||^2 \quad \text{ for all } ||z|| \le r \text{ and every
} h \in \M \, .
\end{equation}
By formula (8.1.11) in \cite{GK} this implies
$$ ||h(\varphi_t(z))+\varphi_t(z)|| \le M_r ||\varphi_t(z)||^2 \le M_r e^{-2
  t} \frac{||z||^2}{(1-||z||)^4}$$
for all $||z|| \le r$ and every $h \in \M$. Therefore,
$$ e^t h(\varphi_t(z))=-e^t \varphi_t(z)+e^t \left(
  h(\varphi_t(z))+\varphi_t(z) \right) \to -F(z) \qquad ( t \to \infty) $$
locally uniformly for $z \in \B^n$ and uniformly for $h \in \M$. Since we also
have $\left[ d(e^t\varphi_t) \right]^{-1} \to \left[ d(F)_z \right]^{-1}$ locally
uniformly in $\B^n$ as $t \to \infty$, we get
$$ L_t(h)=L\left( d(F)_z \cdot \left[ d(\varphi_t)_z \right]^{-1}
  h(\varphi_t)\right)=L \left( d(F)_z \cdot \left[ d(e^t\varphi_t)_z \right]^{-1}
  e^th(\varphi_t)\right) \to -L(F) $$
uniformly for $h \in \M$, so 
$$m(t)=\max \limits_{h \in \M} \Re L_t(h) \to -\Re L(F) \qquad (t \to \infty)
\, .$$
On the other hand, 
$$ m(0)=\max \limits_{h \in \M} \Re L_0(h)=\max \limits_{h \in \M} \Re
L\left(d(F)_z \cdot h \right) \, .$$
Therefore, Theorem \ref{thm:main2} (b) completes  the proof of Theorem \ref{thm:pommerenke}.
\end{proof}

We next show that under the condition that $L$ is not constant on
$\mathcal{S}_n^0$, the continuous linear functionals $L_t$ in Theorem
\ref{thm:main2} are support points of $\M$. First we recall the following definition.

\begin{definition} \label{def:supp}
Let $\mathcal{A} \subseteq \HB$.
A function $G \in \mathcal{A}$ is called a \textit{support point} of $\mathcal{A}$, if there exists a
continuous linear functional $L : \HB \to \C$ such that $\Re L(h) \le \Re L(G)$
for every $h \in \mathcal{A}$ and $L$ is not constant on $\mathcal{A}$.
We denote by $\supp \mathcal{A}$ the set of all support points of $\mathcal{A}$.
\end{definition}

\begin{proposition} \label{prop:lin}
Let $\Phi$ be a complex functional with complex derivative $L$ at $F \in
\mathcal{S}_n^0$ and suppose that $F$ maximizes $\Re \Phi$ over
$\mathcal{S}_n^0$.
Let $G$ be a Herglotz vector field in the class $\M$ with $(\varphi_t):=(\varphi^G_t)$
 such that $F=e^{\tau} \varphi_{\tau}$ for some $\tau \in (0,\infty]$.
Suppose that $L$ is not constant on $\mathcal{S}_n^0$.
Then for any $t \in [0,\tau]$ the continuous linear functional
$$ h \mapsto L_t(h):=L\left(d(F)_z \cdot   \left[d(e^t
  \varphi_t)_z\right]^{-1} \cdot h(\varphi_t) \right)$$
is not constant on $\M$. 
\end{proposition}

\begin{proof} We  show that if $L_t$ is constant on $\M$ for some $t \in
  [0,\tau]$, then $L$ is constant on $\mathcal{S}^0_n$. Hence let $t \in
  [0,\tau]$ such  that $L_t(h)$ is constant on $\mathcal{M}_n$. Let $P : \C^n \to \C^n$
be a polynomial mapping with $P(0)=0$ and $d(P)_0=0$. Then there is a number $\delta>0$
such that $-z+\eps P(z) \in \M$ for every $\eps \in \C$ with $|\eps|<\delta$,
so $L_t(-z+\eps P)=- L_t(z)+\eps L_t(P)$ is constant in $\eps$.
This implies that $L_t(P)=0$. Now let $g \in \HB$ with $g(0)=0$ and
$d(g)_0=0$.
Since $\varphi_t(\B)$ is Runge (see \cite{ABFW}), $g$ is the locally uniform
limit of $(P_k \circ \varphi_t)_k$ for a sequence of polynomials $P_k$ with
$P_k(0)=0$ and $d(P_k)_0=0$. Hence
$$ 0=\lim \limits_{k \to \infty} L_t(P_k)=
L \left(d(F)_z \cdot   \left[d(e^t
  \varphi_t)_z\right]^{-1}  \cdot g\right) \, $$
for all  $g \in \HB$ with $g(0)=0$ and
$d(g)_0=0$. This clearly implies $L(g)=0$ for all such $g$.
Since we can take $g=f-\id$ for any $f \in \mathcal{S}^0_n$, we get
  $L(f)=L(\id)$, $f \in \mathcal{S}^0_n$, so $L$  is constant on
$\mathcal{S}_n^0$.
\end{proof}

\begin{definition} 
A function $F \in \mathcal{S}_n^0$ is called \textit{extremal}, if there exists
a  complex functional $\Phi$ with complex derivative $L$ at $F$, such that
\begin{itemize}
\item[(a)] $\Re \Phi(f) \le \Re \Phi(F)$ for all $f \in \mathcal{S}_n^0$, and
\item[(b)] $L$ is not constant on $\mathcal{S}_n^0$.
\end{itemize}
\end{definition}

\begin{theorem} \label{thm:main3}
Let $G(z,t)$ be a Herglotz vector field in the class $\M$.
Let $e^{\tau} \varphi_{\tau}^G$ be extremal for some $\tau \in
(0,\infty]$. Then $G(\cdot,t) \in \supp \M$ for every $t \in R_G \cap (0,\tau]$.
\end{theorem}

\begin{proof}
This follows immediately from Theorem \ref{thm:main2} (a) and Proposition \ref{prop:lin}.
\end{proof}

Theorem \ref{thm:main3a} now follows directly from Corollary \ref{cor:main}
for $\tau=\infty$ and Lemma \ref{lem:lebesgue}, which shows that the set 
$R_G$ of regular points of any Herglotz vector field $G(z,t)$ in the class
$\M$ has full measure.

\begin{corollary} \label{cor:main}
Let $G(z,t)$ be a Herglotz vector field in the class $\M$ and
assume that
$$ \sup \limits_{z \in \B^n\backslash \{0\}} \Re \langle G(z,T),z/||z||^2 \rangle<0$$
for some $T \in R_G$. Then $e^t \varphi_t \in \mathcal{S}_n^0$ is not extremal
for any $t \in (T,\infty]$. 
\end{corollary}

\begin{proof}
In view of Theorem \ref{thm:main3}, it suffices to show that $h:=G(\cdot,T)$ cannot be
a support point of the class $\M$ on $\B^n$. By assumption, there is a constant $a>0$
such that $\Re \langle G(z,T),z \rangle \le -a ||z||^2$ for all $z \in \B^n$.
This implies that for any polynomial mapping $P : \C^n \to \C^n$ with $P(0)=0$
and $d(P)_0=0$ there is a number $\delta>0$ such that $h+\eps P \in \M$ for every
$\eps \in \C$ with $|\eps|<\delta$. If $h \in \supp \M$, there is a
continuous linear functional $L$ on $\HB$ such that $\max_{g \in \M} \Re L(g)=\Re
L(h)$ and $L$ is not constant on $\M$. We can now argue as in the proof of
Proposition \ref{prop:lin}.
In particular, $\Re L(h+\eps P)
\le \Re L(h)$ for any $|\eps|<\delta$, so $L(P)=0$ for every polynomial mapping $P : \C^n \to \C^n$ with $P(0)=0$
and $d(P)_0=0$. Hence $L=0$ on the set of functions $g \in \HB$ with $g(0)=0$
and $d(g)_0=0$, so $L$ is constant on $\M$, a contradiction.
\end{proof}

We like to end with a remark that extends Lemma 4.3 in \cite{BHKG}.

\begin{remark}
Corollary \ref{cor:main} shows that  if $(f_t)_{t \ge 0}$ is a
parametric representation such that
$$ \inf \limits_{z \in \B^n \backslash \{ 0\}} \Re \bigg\langle \left[d(f_t)_z\right]^{-1} \frac{\partial f_t}{\partial
  t}(z),\frac{z}{||z||^2}\bigg\rangle >0$$
for all $t \in E$, where $E \subseteq \R^+$ is a set of positive measure, then
$f_0$ is not extremal (in particular, 
$f_0 \not \in \supp \mathcal{S}_n^0$). 
\end{remark}

\vfill

Oliver Roth\\
Department of Mathematics\\
University of W\"urzburg\\
Emil Fischer Stra{\ss}e 40\\
97074 W\"urzburg\\
Germany\\
roth@mathematik.uni-wuerzburg.de

\end{document}